\documentclass{amsart}
\usepackage{amsmath}
\usepackage[inline]{enumitem}
\usepackage{amsfonts}
\usepackage{amssymb}
\usepackage[all]{xy}
\usepackage{mathrsfs}
\usepackage[usenames,dvipsnames]{color}
\usepackage[colorlinks]{hyperref}

\title{Stable maps to rational curves and the relative Jacobian}
\author{Steffen Marcus and Jonathan Wise}
\date{\today}

\newcommand{\oMct}{\overline{M}\vphantom{M}^{\rm ct}}
\newcommand{\tsD}{\widetilde{\mathscr{D}}}
\newcommand{\fP}{\mathfrak{P}}

\newcommand{\sE}{\mathscr{E}}
\newcommand{\tensor}{\mathop{\otimes}}
\DeclareMathOperator{\uIsom}{\underline{Isom}}
\DeclareMathOperator{\uHom}{\underline{Hom}}
\newcommand{\sU}{\mathscr{U}}
\newcommand{\sT}{\mathscr{T}}
\newcommand{\sV}{\mathscr{V}}
\newcommand{\oM}{\overline{M}\vphantom{M}}
\newcommand{\sP}{\mathscr{P}}
\newcommand{\tsP}{\widetilde{\sP}}
\newcommand{\bP}{\mathbf{P}}
\newcommand{\Gm}{\mathbf{G}_m}
\newcommand{\sD}{\mathscr{D}}
\newcommand{\BGm}{\mathcal{B} \Gm}
\newcommand{\fM}{\mathfrak{M}}
\newcommand{\fJ}{\mathfrak{J}}
\newcommand{\fZ}{\mathfrak{Z}}
\newcommand{\tP}{\widetilde{P}}

\newcommand{\cO}{\mathcal{O}}
\newcommand{\sA}{\mathscr{A}}

\newcommand{\bA}{\mathbf{A}}
\newcommand{\bx}{\mathbf{x}}
\newcommand{\sL}{\mathscr{L}}
\newcommand{\sM}{\mathscr{M}}
\newcommand{\id}{\mathrm{id}}
\newcommand{\sQ}{\mathscr{Q}}

\newcommand{\DR}{{\bf DR}}
\newcommand{\DRsp}{\mathrm{DR}}

\newcommand{\targets}{\sV}

\newtheorem*{theorem}{Theorem}
\newtheorem{proposition}{Proposition}
\newtheorem{lemma}[proposition]{Lemma}
\newtheorem{corollary}{Corollary}
\numberwithin{corollary}{proposition}
\numberwithin{proposition}{section}
\theoremstyle{remark}

\theoremstyle{definition}
\newtheorem*{construction}{Construction}

\begin{document}

\begin{abstract} 
We consider two cycles on the moduli space of compact type curves and prove that they coincide.  The first is defined by pushing forward the virtual fundamental classes of spaces of relative stable maps to an unparameterized rational curve, while the second is obtained as the intersection of the Abel section of the universal Jacobian with the zero section.  Our comparison extends results of \cite{CMW} where the same identity was proved over on the locus of rational tails curves.
\end{abstract}

\maketitle
\section{Introduction}
Let $\oM(\sP/\BGm)$ denote the moduli space of relative stable maps to a ``rubber'' target---that is, to a chain of unparameterized rational curves with marked points $0$ and $\infty$ on the end components---and $[ \oM(\sP/\BGm)]^{\rm vir}$ its virtual fundamental class.  Write $\oM$ for the stack parameterizing stable marked curves equipped with a neutral integral weighting, by which we mean a weighting of its marked points by a vector of integers $\bx$ of sum zero.  Denote by $\oM^{\rm ct}$ and $\oM^{\rm rt}$ respectively the loci of compact type curves and rational tails curves in $\oM$.  We employ the same decorations to refer to the appropriate substacks of $\oM(\sP/\BGm)$.

Our goal is to compare two natural cycles on $\oM^{\rm ct}$.  One one hand, there is a morphism $\Pi : \oM^{\rm ct}(\sP/\BGm) \rightarrow \oM^{\rm ct}$ through which we can push forward the virtual fundamental class.  It is defined by sending a stable map to the Deligne--Mumford stabilization of its source curve.  On the other hand, we have a canonical section of the relative Jacobian of the universal curve over $\oM^{\rm ct}$, given generically by $\sigma:(C,\{p_i\})\mapsto \sum_i x_i[p_i]$,  and we may take the intersection of the image of this section and the zero section to obtain a cycle class $\DR$.  We write $\DRsp$ for the scheme theoretic intersection, on which $\DR$ is supported.  These cycles have gone---somewhat inconsistently---by various names in the literature, including the \emph{double Hurwitz class} and the \emph{double ramification cycle}.

In~\cite{CMW}, we showed that these two classes agree \textit{over the rational tails locus}. 
Here we will extend this result to \emph{compact type} curves:
\begin{theorem}
\begin{equation*}
\Pi_\ast [ \oM^{\rm ct}(\sP/\BGm) ]^{\rm vir} = \DR.
\end{equation*}
\end{theorem}

The moduli space $\oM(\sP/\BGm)$ carries discrete invariants given by the genus $g$ of the source curve and the $n$ marked pre-images of $0$ and $\infty$ along with their weights, encoded in a vector $\bx$.  Fixing $g$ and $\bx$ distinguishes a component $\oM_g(\sP/\BGm;\bx)$ parameterizing stable maps to an unparameterized rational curve relative to two marked points, with ramification profiles prescribed by the positive and negative parts of $\bx$ respectively.  This gives a compactification of the space of Hurwitz covers associated to the same data, and admits a branch divisor morphism, the degree of which is the double Hurwitz number $H_g(\bx)$.  Double Hurwitz numbers are piecewise polynomials, controlled by a rich combinatorial structure described in \cite{GJV:DHN,SSV,CJM}.  The cycle $\Pi_\ast[\oM_g(\sP/\BGm;\bx)]^{\rm vir}$ was introduced in~\cite{GVearly,GJV} as a geometric tool by which to analyze the relationship between double Hurwitz numbers and the geometry of $\oM_{g,n}$.  It was shown in \cite{FP} to be tautological.  Most recently, Buryak--Shadrin--Spitz--Zvonkine \cite{BSSZ} have computed the intersection numbers of any monomial of psi classes against this cycle.

Independently, Eliashberg has sought a better understanding of the double ramification cycle for use in symplectic field theory~\cite{EGH}.  This has led various authors to the construction of $\DR_{g,n}$ described above via intersections on the relative Jacobian.  Techniques of Hain \cite{H} and, independently, Grushevsky--Zakharov~\cite{GZ1}, provide a remarkable computation as a polynomial of degree $2g$ in the $x_i$ with coefficients in the tautological ring.  The limitations inherent in using the relative Jacobian, however, restrict $\DR_{g,n}$ to the compact type locus.  Recent advances of Grushevsky--Zakharov \cite{GZ2} have extended this slightly.  We hope to compare these parallel approaches to the double ramification cycle on their common domain of definition; the present theorem provides this comparison up to compact type.

We happily acknowledge some very helpful comments and suggestions of R.\ Pandharipande and D.\ Zvonkine, not least their encouragement that we return to this question.

\section{Rubber targets}\label{sec:expansions}
\label{sec:targets}

In this section we briefly outline the construction of the universal family of expansions for stable maps to rubber targets.

\subsection{The rubber rational curve}\label{sec:rubberP1}
Let $\sP = [ \bP^1 / \Gm ]$ where $\Gm$ acts on $\bP^1$ by scaling the first coordinate.  We can also realize $\sP$ as the complement of the closed point in $\sA^2$, where $\sA = [\bA^1 / \Gm]$.  Each of $\sP$, $\sA$ and $\bP^1$ have modular interpretations:
\begin{itemize}
\item $\sA$ is the moduli space of pairs $(L, s)$ where $L$ is a line bundle and $s$ is a section of $L$;
\item $\bP^1$ is the moduli space of triples $(L, s, t)$ where $s$ and $t$ are sections of the line bundle $L$ that do not vanish simultaneously;
\item $\sP$ is the moduli space of tuples $(L, s, M, t)$ where $L$ and $M$ are line bundles and $s$ and $t$ are, respectively, sections of $L$ and $M$ that do not vanish simultaneously.  
\end{itemize}
We write $\sL$ and $\sM$ for the universal line bundles over $\sP$.  The vanishing loci of $s$ and $t$ are divisors on $\sP$, which we denote $\sD_-$ and $\sD_+$.

The quotient map $\bP^1 \rightarrow \sP$ realizes $\bP^1$ as a $\Gm$-torsor over $\sP$.  

\begin{proposition}
Over $\sP$, the $\Gm$-torsors $\bP^1$ and $\uIsom(\sL,\sM)$ 
are isomorphic.
\end{proposition}
\begin{proof}
Recall that $\Gm$ acts on a point $(L,s,t)$ of $\bP^1$ by scaling $s$.  We have the map
\begin{equation*}
\bP^1 \rightarrow \sP : (L, s, t) \mapsto (L, s, L, t) .
\end{equation*}  
An object of $\uIsom(\sL,\sM)$ is a tuple $(L, s, M, t, \varphi)$ where $(L,s,M,t) \in \sP$ and $\varphi : L \rightarrow M$ is an isomorphism of line bundles, and $\Gm$ acts by scaling $\varphi$.  We have a $\sP$-morphism
\begin{equation*}
\bP^1 \rightarrow \uIsom(\sL,\sM) : (L,s,t) \mapsto (L, s, L, t, \id_L) .
\end{equation*}
The action of $\lambda \in \Gm$ gives
\begin{equation*}
\lambda . (L, s, L, t, \id_L) = (L, s, L, t, \lambda \id_L) \simeq (L, \lambda s, L, t, \id_L)
\end{equation*}
where the isomorphism is scaling the first copy of $L$ by $\lambda^{-1}$.  This exhibits our desired isomorphism.
\end{proof}

The $\Gm$-torsor $\uIsom(\sL,\sM)$ has an associated line bundle \[\uHom(\sL,\sM) = \cO_{\sP}(\sD_+ - \sD_-),\] giving a map $\sP \rightarrow \BGm$ whose fiber is $\bP^1$.  More generally, if $S \rightarrow \BGm$ is a morphism classifying a line bundle $L$ over $S$, then the pullback of $\sP$ is the projective completion of the line bundle $L$.

\subsection{Expansions}\label{sec:expansions}

An expansion of $\sP$ is a chain
\begin{equation} \label{eqn:6}
\tsP = \sP \mathop{\amalg}_{\sD_- \simeq \sD_+} \sP \mathop{\amalg}_{\sD_- \simeq \sD_+} \cdots \mathop{\amalg}_{\sD_- \simeq \sD_+} \sP .
\end{equation}
The isomorphism $\sD_- \simeq \sD_+$ used for gluing is the canonical one by which $N_{\sD_- / \sP}$ corresponds to $N_{\sD_+ / \sP}^\vee$.  An expansion $\tsP$ has an associated chain of rational curves
\newlength{\inftywidth}
\settowidth{\inftywidth}{$\scriptstyle\infty$}
\newsavebox{\gluesub}
\savebox{\gluesub}{$\!\scriptstyle\infty\simeq\!$\makebox[\inftywidth]{$\scriptstyle0$}}
\begin{equation*}
\tP = \bP^1\!\mathop{\amalg}_{\usebox{\gluesub}} \bP^1\!\mathop{\amalg}_{\usebox{\gluesub}} \cdots \mathop{\amalg}_{\usebox{\gluesub}} \bP^1 \in\fM_{0,2}^{\rm ss}
\end{equation*}
with a quotient map $\tP\to\tsP$.
There is a universal flat family of expansions of $\sP$ (see \cite[Section~4]{ACFW}).  We use $\sP^{\exp}$ for the total space of the universal family and $\targets$ for its base.  This space is closely related to the universal family of chains of rational curves:  there is a canonical isomorphism $\fM_{0,2}^{\rm ss} \simeq \targets \times \BGm$.  

There is an \'etale cover of $\targets$ by stacks $\sA^n$ (see \cite[Section~8.3]{ACFW}).  The expansion of $\sP$ inducing the map $\sA^n \rightarrow \targets$ is the locally closed substack  $\sU_n \subset \sA^{2n + 2}$, with coordinates $(\sigma_0, \sigma_1, \ldots, \sigma_n, \tau_1, \ldots, \tau_n, \tau_{n+1})$ defined by the non-vanishing of the pair $(\sigma_i, \tau_j)$ for $i < j$ and the vanishing of the product $\sigma_i \tau_j$ for $i \geq j$.%
\footnote{We say that an object $(L,s)$ of $\sA$ vanishes if it is isomorphic to $(\cO,0)$.}
These coordinates are taken using the modular description of $\sA$ above, so each $\sigma_i = (L_i, s_i)$ where $L_i$ is a line bundle and $s_i$ is a section of $L_i$, and similarly for $\tau_i = (M_i, t_i)$.  The projection to the base is induced by restriction from
\begin{equation*}
\sA^{2n + 2} \rightarrow \sA^n : (\sigma_0, \ldots, \sigma_n, \tau_1, \ldots, \tau_{n+1}) \mapsto (\sigma_1 \tau_1, \ldots, \sigma_n \tau_n) .
\end{equation*}

To a function%
\footnote{We use $[n]$ to denote the set $\{ 1, \ldots, n \}$.}
$[n] \rightarrow [m]$ we may associate a map $f : \sA^n \rightarrow \sA^m$ where, if $\alpha_i$ denotes the $i$-th coordinate on $\sA^m$ or $\sA^n$, we have 
\begin{equation} \label{eqn:5}
f^\ast(\alpha_i) = \prod_{j \in f^{-1}(\alpha_i)} \alpha_j
\end{equation}
A morphism $f : \sA^n \rightarrow \sA^m$ commutes with the projections to $\targets$ if and only if it is associated as above to an order preserving injection:%
\begin{equation*}
\xymatrix{
\sA^n\ar[rr]\ar[dr]&&\sA^m\ar[dl]\\
&\targets&}
\end{equation*} 
These morphisms constitute a presentation of $\targets$ (see loc.\ cit.\ for more details).  On the level of universal families, the map induced by $f$ is given by the same formula~\eqref{eqn:5} with $\alpha$ replaced by $\sigma$ or $\tau$ and the convention that $u^{-1}(0) = \{ 0 \}$ and $u^{-1}(m+1) = \{ n+1 \}$.

\subsection{The automorphism group of an expansion}
\label{sec:aut-exp}

Let $\tsP$ be as in~\eqref{eqn:6}.  An automorphism of $\tsP$ consists of maps $\sP \rightarrow \sP$ preserving the divisors $\sD_-$ and $\sD_+$---of which there is exactly one---together with compatibility data for each node.  Let $\sE \subset \tsP$ be such a node, joining components $\tsP'$ and $\tsP''$.  To glue automorphisms of $\tsP'$ and $\tsP''$ along $\sE$, we must choose an isomorphism between the induced maps (the identity) $\sE \rightarrow \sE$.  Since $\sE \simeq \BGm$, this amounts to the selection of an element of $\Gm$.  This proves the following lemma.

\begin{lemma}
Let $\tsP$ be an expansion of $\sP$ as in~\eqref{eqn:6} having $n$ nodes.  Then the automorphism group of $\tsP$ is isomorphic to $\Gm^n$.
\end{lemma}

\noindent These automorphisms, which are invisible on the underlying topological space of $\tsP$, have sometimes been called \textit{ghost automorphisms}.

We may understand the ghost automorphisms in somewhat more concrete terms by considering the functor of points of $\tsP$.  An $S$-point of $\tsP$ is 
\begin{enumerate}
\item a system of $n$ line bundles and sections $(L_i, s_i)$ and $(M_i, t_i)$, and
\item isomorphisms $(L_i \tensor M_i, s_i \tensor t_i) \simeq (\cO_S, 0)$, 
\end{enumerate}
such that 
\begin{enumerate}[resume*]
\item for all $i$ the sections $s_i$ and $t_{i+1}$ do not vanish simultaneously.\end{enumerate}
The action of $\lambda$ from the $i$-th factor in $\Gm^n$ on these data is to scale the isomorphism $L_i \tensor M_i \rightarrow \cO_S$.

\subsection{The universal family of chains of rational curves}  Using the presentation of $\targets$ in Section~\ref{sec:expansions}, we will now describe the universal family of chains of rational curves.  

Note that projection to the coordinates $(\sigma_0, \tau_{n+1})$ gives a map $\sA^{2n+2} \rightarrow \sA^2$ whose restriction to $\sU_n$ factors through $\sP$.  Since the coordinates $\sigma_0$ and $\tau_{n+1}$ are fixed by the morphisms $\sA^n \rightarrow \sA^m$ in the diagram above, this map descends to give a map $\sP^{\exp} \rightarrow \sP$.  We denote the pre-images of $\sD_-$ and $\sD_+$ by $\sD_-^{\exp}$ and $\sD_+^{\exp}$.  When $\tsP$ is an expansion of $\sP$, we write $\tsD_-$ and $\tsD_+$ for the pre-images of $\sD_-^{\exp}$ and $\sD_+^{\exp}$ in $\tsP$.  Also, there is an obvious map from $\sP^{\exp}$ to $\BGm$ given by pulling back the line bundle $\uHom(\sL,\sM)$ on $\sP$.

We can construct a second map $\sP^{\exp} \rightarrow \BGm$ determined by the line bundle \[\sQ = \uHom(L_1 \tensor \cdots \tensor L_n, M_1 \tensor \cdots \tensor M_n).\]  This line bundle descends to a line bundle on $\sP$ and gives a map $\sP \rightarrow \BGm$.  Let $Q$ be the associated $\Gm$-torsor $\uIsom(L_1 \tensor \cdots \tensor L_n, M_1 \tensor \cdots \tensor M_n)$.  Note that $\sQ$ has degree zero on every component of every fiber of $\sU_n$.  The line bundle $\sQ$ and $\Gm$-torsor $Q$ should be regarded as the expansions to $\sP^{\exp}$ of the line bundle and $\Gm$-torsor on $\sP$ discussed in Section~\ref{sec:rubberP1}.%

We analyze the total space of $Q$.  For each $1 \leq i \leq n+1$ we have a projection $\sU_n \rightarrow \sP$ on to the $i$-th factor:  $(\sigma_0, \ldots, \sigma_n, \tau_1, \ldots, \tau_{n+1}) \mapsto (\sigma_{i-1}, \tau_i)$.  This map has a section sending $\sP$ to the vanishing locus of $\sigma_{i}$ and $\tau_{i-1}$.  In explicit terms, the section sends $(\sigma, \tau)$ to $(\sigma_0, \ldots, \sigma_n, \tau_1, \ldots, \tau_{n+1})$ with%
\footnote{For notational simplicity, we write $0$ and $1$, respectively, for the trivial line bundle with section $0$ or $1$.}
\begin{align*}
\sigma_j & = \begin{cases} 1 & j < i-1 \\ \sigma & j = i-1 \\ 0 & j > i-1 \end{cases} &
\tau_j & = \begin{cases} 0 & j < i \\ \tau & j = i \\ 1 & j > i \end{cases} .
\end{align*}

\noindent Restricting $Q$ via this inclusion, we get a copy of $\uIsom(\sL,\sM) = \bP^1$ via the standard projection to $\sP$.  Then $Q$ is a flat family over $\sA^n$ whose fibers are chains of rational curves and the family $Q \rightarrow \fM_{0,2}^{\rm ss} = \targets \times \BGm$ is the universal family of chains of rational curves.

\subsection{Lifting maps to rubber targets}
\label{sec:lifting}

Let $C\to S$ be a pre-stable curve, and suppose we have a stable map $C\to\tsP$ to an expansion of $\sP$.  These data determine a diagram of solid lines
\begin{equation} \label{eqn:2} \vcenter{\xymatrix{
& \tP \ar[d] \\
C \ar[r] \ar[d]_\pi \ar@{-->}[ur] & \tsP \ar[d] \\
S \ar@{-->}[r] & \BGm.
}} \end{equation}
Here, $\tP$ is the associated expansion of $\bP^1$, i.e., the restriction of $\Gm$-torsor $Q$ to $\tsP$.  Note, however, that the map $\tsP\to \BGm$ is the restriction of our map $\sP^{\exp} \rightarrow \BGm$ determined by the line bundle $\sQ$, thus again represents the restriction of $Q$.  Hence the obstruction to completing the top triangle is precisely the same as the obstruction to completing the bottom square:  it is the class in $R^1 \pi_\ast \cO_C^\ast$ of the restriction of the torsor $Q$ to $C$.

Note furthermore that the restriction of $Q$ to $C$ induces a section of the relative Picard stack $\fP$ of $C/S$.  Giving a lift of diagram~\eqref{eqn:2} amounts to factoring this section through the zero locus of $\fP$.

\section{Moduli spaces}

In this section we describe the two moduli spaces of stable maps in question.  They fit naturally into a cartesian diagram with the universal relative Jacobian and its zero section,  diagram \eqref{diag:cart} below.  

\subsection{Stable maps to rubber targets}
Let $\oM(\sP/\BGm)$ be the stack of rubber maps to chains of rational curves.  An object over a scheme $S$ is a commutative diagram
\begin{equation} \label{eqn:1} \vcenter{\xymatrix{
C \ar[r]^{f} \ar[d] & \sP^{\exp} \ar[d] \\
S \ar[r] & \targets \times \BGm
}} \end{equation}
in which
\begin{enumerate}[label=\textbf{R\arabic{*}}]
\item \label{cond:nodal} $C$ is a pre-stable curve over $S$,
\item \label{cond:predef} $f$ is predeformable, and
\item \label{cond:stable} the automorphism group of all the data is finite.
\end{enumerate}

Let $\oM(\sP)$ be the stack of diagrams
\begin{equation} \label{eqn:3} \vcenter{\xymatrix{
C \ar[r] \ar[d] & \sP^{\exp} \ar[d] \\
S \ar[r] & \targets
}} \end{equation}
with the same conditions as well as the following:%
\footnote{This condition was incorrectly omitted in \cite{CMW}.}
\begin{enumerate}[resume*]
\item \label{cond:degree} Let $\sQ$ be the line bundle defined in Section~\ref{sec:targets}.  Then $f^\ast \sQ$ has degree zero on every irreducible component of every fiber of $C$.
\end{enumerate}
Note that \ref{cond:degree} holds automatically on $\oM(\sP/\BGm)$.

\begin{lemma}
Given a predeformable diagram~\eqref{eqn:3} the locus in $S$ where the diagram satisfies~\ref{cond:degree} is open.
\end{lemma}
\begin{proof}
By local finite presentation, we may assume $S$ is of finite type.  We argue that the locus in question is constructible and stable under generization.

We may stratify $S$ into locally closed subsets on which the dual graphs of $C$ and of $\sP$ are constant.  Then the degree of $f^\ast \sQ$ on each irreducible component of $C$ is a locally constant function on each stratum, thus a constructible function on $S$.

Now consider a geometric point $s$ of $S$ that is a specialization of $\eta$.  Let $C'$ be the closure of an irreducible component of the generic fiber.  Then the degree of $f^\ast \sQ$ on $C'_\eta$ is the sum of the degrees of $f^\ast \sQ$ on the irreducible components of $C'_s$, hence is zero.
\end{proof}

Note that if we use the notation $\fM(\,\cdots)$ to refer to the stacks obtained by omitting the finiteness requirement~\ref{cond:stable} on automorphism groups, then we have an obvious map
\begin{equation*}
\fM(\sP/\BGm) \rightarrow \fM(\sP)
\end{equation*}
by forgetting the map $S \rightarrow \BGm$.  In fact, the following lemma shows that this map carries stable objects to stable objects, so provides us with a map
\begin{equation*}
\oM(\sP/\BGm) \rightarrow \oM(\sP) .
\end{equation*}


\subsection{Infinitesimal automorphisms}

In this section, we will study the infinitesimal automorphism group of a predeformable map $f : C \rightarrow \tsP$.  By definition, an automorphism of such a map is a commutative diagram
\begin{equation*} \xymatrix{
C \ar[r] \ar[d]_<>(0.5)f & C \ar[d]^<>(0.5)f \\
\tsP  \ar[r] & \tsP .
} \end{equation*}
We shall assume that $(C, \tsP, f)$ satisfies~\ref{cond:nodal}, \ref{cond:predef}, and \ref{cond:degree}.  One obvious source of infinitesimal automorphisms is from unstable components of $C$ that are contracted by $f$.  As these are easy to understand, we will be more interested in the automorphisms of $(C, \tsP, f)$ that restrict to nontrivial automorphisms of $\tsP$.

Recall from Sections~\ref{sec:expansions} and \ref{sec:aut-exp} that the map $f$ is specified by $2n + 2$ line bundles and sections
\begin{equation*}
(L_0, s_0), \ldots, (L_n, s_n), (M_1, t_1), \ldots, (M_{n+1}, t_{n+1})
\end{equation*}
as well as isomorphisms $\varphi_i : L_i \tensor M_i \xrightarrow{\sim} \cO_C$ for all $i$, satisfying the following conditions:
\begin{enumerate}
\item $s_i$ and $t_j$ do not vanish simultaneously for $i < j$,
\item $s_i \tensor t_i = 0$ for all $i$,
\item $s_i$ and $t_i$ do not vanish identically on any irreducible component of $C$, and
\item the orders of vanishing of $t_i$ and $s_{i}$ are the same on opposite sites of a node of $C$.
\end{enumerate}
The first two conditions come from the functor of points of $\tsP$; the last two come are predeformability, which says in particular that no irreducible component of $C$ is carried by $f$ into a node of $\tsP$ (namely, the vanishing locus of $s_i$ and $t_i$).

Furthermore, recall that the automorphism group of $\tsP$ is $\Gm^n$ with the $i$-th factor acting by scaling $\varphi_i$.  For a fixed infinitesimal automorphism $\boldsymbol{\lambda} = (\lambda_1, \ldots, \lambda_n) \in \Gm^n$ of $\tsP$, consider the infinitesimal automorphisms of the triple $(C, \tsP, f)$ restricting to $\boldsymbol{\lambda}$ on $\tsP$.  To give such an automorphism, we must give an automorphism $\sigma$ of $C$ and isomorphisms 
\begin{gather*}
u_i : \sigma^\ast (L_i, s_i) \xrightarrow{\sim} (L_i, s_i) \\
v_i : \sigma^\ast (M_i, t_i) \xrightarrow{\sim} (M_i, t_i)
\end{gather*}
such that the diagrams
\begin{equation*} \xymatrix{
\sigma^\ast L_i \tensor \sigma^\ast M_i \ar[r]^<>(0.5){\sigma^\ast \varphi_i} \ar[d]_{u_i \tensor v_i} & \cO \ar@{=}[d] \\
L_i \tensor M_i \ar[r]^<>(0.5){\lambda_i \varphi_i} & \cO
} \end{equation*}
commute for all $i$.  It follows from~\ref{cond:predef} and~\ref{cond:degree} that none of the $s_i$ vanish identically on $C$.  Therefore the maps $u_i$ are uniquely determined from $\sigma$, provided they exist.  The same applies to the $v_i$.

Consider an irreducible component $D$ of $C$.   If $D$ is not semistable then $\sigma$ acts trivially on $D$, in which case we set $\mu_D = 1$.  Otherwise $\sigma$ acts by scaling and we set $\mu_D$ equal to the scaling factor.  If $i$ is the unique index such that $s_i$ and $t_{i+1}$ do not vanish identically on $D$ then $\sigma^\ast (L_i, s_i) \simeq (L_i, \mu_D^{r_D} s_i)$ and $\sigma^\ast (M_{i+1}, t_{i+1}) \simeq (M_{i+1}, \mu_D^{-{r_D}} t_{i+1})$, where $r_D$ is the common order of vanishing of both $s_i$ and $t_{i+1}$ on $D$.  It follows that $\mu_D^{r_D}$ must be the same for all irreducible components $D$ of $C$ with the same image in $\tsP$.  We may therefore introduce the notation $\mu_i$ for the common value of $\mu_D^{r_D}$ on all irreducible components on which $s_i$ vanishes but does not vanish identically.%

The map $\sigma^\ast \varphi_i \circ (u_i \tensor v_i)^{-1} : L_i \tensor M_i \rightarrow \cO$ is $\mu_i \mu_{i-1}^{-1} \varphi_i$.  We therefore get the equation $\mu_i \mu_{i-1}^{-1} = \lambda_i$.  The $\mu_i$ therefore determine the $\lambda_i$, and, except for the geometric constraints from $C$ already described, the $\mu_i$ may be arbitrary.

We may finally conclude:
\begin{lemma}
Let $f : C \rightarrow \tsP$ be a predeformable morphism from a connected, nodal curve to an expansion $\tsP$ of $\sP$.  Assume that $f$ satisfies~\ref{cond:predef} and~\ref{cond:degree}.  Then $(C, \tsP, f)$ has an infinite automorphism group if and only if either
\begin{enumerate}
\item there is an irreducible component of $C$ that is contracted by $f$ and has genus $g$ with $n$ special points and $2g - 2 + n \leq 0$, or
\item there is an irreducible component of $\tsP$ whose pre-image in $C$ consists entirely of semistable components.
\end{enumerate}
\end{lemma}

\begin{corollary}
A point of $\fM(\sP/\BGm)$ is satisfies~\ref{cond:stable} if and only if its image in $\fM(\sP)$ does.
\end{corollary}

\subsection{The cartesian diagram}
Let $\fM$ be the Artin stack of Deligne-Mumford pre-stable curves.  The moduli spaces $\oM(\sP/\BGm)$ and $\oM(\sP)$ live naturally over $\fM$ by forgetting all but the source curve of a stable map.  Denote by $\fJ$ the universal Jacobian over $\fM$ and $\fZ$ the image of its zero section, which is of course isomorphic to $\fM$.

It is immediate from the definitions that the obstruction to lifting an $S$-point of $\oM(\sP)$ to an $S$-point of $\oM(\sP/\BGm)$ is the map $C \rightarrow \sP^{\exp} \rightarrow \BGm$:  this composition factors through $S$ if and only if a lift exists (see Section~\ref{sec:lifting}).  We therefore have a cartesian diagram
\begin{equation} \label{diag:cart} \vcenter{\xymatrix{
\oM(\sP/\BGm) \ar[r] \ar[d] & \fZ \ar[d] \\
\oM(\sP) \ar[r] & \fJ.
}} \end{equation}  
Note that the map $\oM(\sP) \rightarrow \fJ$ takes values in $\fJ$ and not merely the universal Picard because of condition~\ref{cond:degree}.  Since $\fZ$ is pulled back from the zero section of the universal relative Jacobian over the moduli space of stable curves, we also obtain a cartesian diagram
\begin{equation} \label{eqn:4} \vcenter{\xymatrix{
\oM(\sP/\BGm) \ar[r] \ar[d] & Z \ar[d] \\
\oM(\sP) \ar[r] & J
}} \end{equation}
where $J$ is the universal relative Jacobian over $\oM$ and $Z \simeq \oM$ is the image of its zero section. 

\subsection{Virtual classes} 

\begin{theorem}[{\cite[Proposition~3.5 and Corollary~4.10]{CMW}}]
The following cycle classes on $\oM(\sP/\BGm)$ coincide:
\begin{enumerate*}[label=(\roman{*})]
\item the virtual fundamental class, as studied in~\cite{GV},
\item the relative virtual fundamental class over $\oM(\sP)$, as defined in~\cite{CMW}, and
\item the Gysin pullback of the cycle $Z \subset J$ in diagram~\eqref{eqn:4}.
\end{enumerate*}
\end{theorem}

We regret that even though the arguments of~\cite[Section~4]{CMW} are valid for all nodal curves, some of the notation employed there was specific to the rational tails case.  We have adopted better notation here, and we have also abbreviated some other symbols.  To bring the notation of loc.\ cit.\ in line with the present paper, the following substitutions should be made:

\begin{center}
\begin{tabular}{c|c}
\cite{CMW} & this paper \\
\hline\\[-1.5ex]
$\widetilde{\sT}$ & $\sV \times \BGm$ \\ [1.5ex]
$\sT^2$ & $\sV$ \\[1.5ex]
$\widetilde{\sP}$ & $\sP^{\exp}$ \\[1.5ex]
$\oM_{\rm rel}$ & $\oM$\\
\end{tabular}
\end{center}
With these substitutions, the proof goes through otherwise unchanged.  Note that $\sV$ is not isomorphic to $\sT^2$---the latter parameterizes expansions of $\sP$ with a distinguished component---though there is an \'etale map $\sT^2 \rightarrow \sV$.%
\footnote{The use of $\sT^2$ instead of $\sV$ was the technical mistake that limited us to rational tails in \cite{CMW}.}

\section{Proof of the main theorem}

We now prove our comparison theorem over the compact type locus.  There are two steps.  The first is the construction of a section $\oM^{\rm ct} \rightarrow J$ of the universal relative Jacobian through which our map $\oM(\sP)\to \fJ$ factors.  The second step is a pullback of diagram~\eqref{diag:cart} via this factorization and an application of Costello's theorem~\cite[Theorem~5.0.1]{Costello}.  
\subsection{The Abel map}  

We recall the construction of the Abel map over $\oM^{\rm ct}$.  Given a curve $C$ with neutral integral weights $\bx$ we show that $\cO_C(\sum x_i p_i)$ can be adjusted by a sum of boundary divisors (of the universal compact type curve) to have degree zero on every component of $C$.  This yields a section of the universal relative Jacobian of $\oM^{\rm ct}$.

\begin{construction}
Suppose that $C$ is a stable curve of compact type and $L$ a line bundle on $C$.  By twisting with boundary divisors, we construct from $L$ a line bundle that has degree zero on every component of $C$. 

We proceed by induction on the number of components of $C$, always ensuring that the union of the components of $C$ on which $L$ has nonzero degree is a subtree $T$ of the dual graph of $C$.  If $T$ is reduced to a single component of $C$ then we are done, since the total degree of $L$ is zero; we can therefore assume that $T$ has at least two vertices.

Choose a leaf $C_0$ of $T$, and let $x$ be the node at which $C_0$ is joined to the rest of $T$.  Since $C$ is of compact type, there is a canonical line bundle $M$ on $C$ with section $s$ such that $s$ vanishes exactly along the connected component of $C \smallsetminus \{ x \}$ containing $C_0$.  Therefore $M \big|_{C_0}$ has degree $-1$, and $M$ has degree zero on every component of $T$ other than the one meeting $C_0$.  But if $d = \deg L \big|_{C_0}$ then $M^{\tensor d} \tensor L$ has degree zero on every component in $T \cup \{ C_0 \}$.  This completes the inductive step.
\end{construction}

This construction makes sense in an \'etale neighborhood of a point in the base of a versal family of curves of compact type.  Moreover, the line bundle constructed this way is unique up to tensoring with a line bundle pulled back from the base.  Indeed, suppose we had found two line bundles $L_1$ and $L_2$ by proceeding as above in different orders through the components of $C$.  Then $L_1 \tensor L_2^\vee$ is a sum of boundary components in $C$ that has degree zero on every component of the central fiber of $C$.  It is now easy to see that $L_1 \tensor L_2^\vee$ must be pulled back from the base.  Indeed, all boundary divisors of $C$ appearing in $L_1 \tensor L_2^\vee$ must occur with the same multiplicity.

\subsection{Comparison of classes} Given $C$ in $\oMct$ with neutral integral weights $\bx$ of its markings, the construction described above gives a line bundle $L$ that differs from  $\cO_C(\sum x_i p_i)$ by a sum of boundary components and has degree zero on every component.  This determines a section of the universal Jacobian over $\oMct$.

Consider the map $\oMct(\sP) \rightarrow \oMct$ sending a map $f : C \rightarrow \tsP$ to the stabilization of $C$, marked with $D = f^{-1}(\sD_+^{\exp}) - f^{-1}(\sD_-^{\exp})$ and corresponding weights according to the multiplicity of the divisors.  By composition with the construction above, we obtain a map
\begin{equation*} 
\oMct(\sP) \rightarrow \oMct \rightarrow \fJ .
\end{equation*}
On the other hand, the map $C \rightarrow \sP^{\exp} \rightarrow \BGm$ (the latter map given by the torsor $Q$) gives another map $\oM(\sP) \rightarrow \fJ$.  But recall that the line bundle $\sQ$ on $\sP$ was constructed by adding a sum of boundary divisors to $\sD_+^{\exp} - \sD_-^{\exp}$.  Therefore $f^\ast \sQ$ differs from $D$ by a sum of boundary divisors and the two maps $C \rightarrow \BGm$ must coincide.  We conclude that we have a cartesian diagram
\begin{equation*} \xymatrix{
\oMct(\sP / \BGm) \ar[r] \ar[d] & \DRsp \ar[r] \ar[d] & \fZ \ar[d] \\
\oMct(\sP) \ar[r] & \oMct \ar[r] & \fJ
} \end{equation*}
where $\DRsp = \oMct \mathop{\times}_{\fJ} \fZ$.

The bottom horizontal arrow on the left is birational by \cite[Proposition~3.4]{CMW}.%
\footnote{The reader who would prefer not to rely on orbifold stable maps here may replace them with logarithmic stable maps in the argument of loc.\ cit.}
We may now apply Costello's theorem~\cite[Theorem~5.0.1]{Costello} to complete the proof.

\bibliographystyle{amsalpha}
\bibliography{rubjac}

\end{document}